\title{\vspace{-1cm}Regularity and drift by Osgood vector fields}
\date{\today}
\documentclass[12pt]{article}
\usepackage{authblk}
\usepackage{blindtext}
\usepackage{amsfonts}
\usepackage{amsmath}
\usepackage{amsthm}
\usepackage{graphicx}
\usepackage{hyperref}
\usepackage{slashed}
\usepackage{comment}
\usepackage[margin=2cm]{geometry}
\usepackage{cancel}
\usepackage{collectbox}


\newtheorem{theorem}{Theorem}
\newtheorem{lemma}{Lemma}
\newtheorem{proposition}{Proposition}
\newtheorem{remark}{Remark}

\newtheorem{definition}{Definition}

\author{Joonhyun La}
\begin{document}
\maketitle
\abstract{We consider the problem of loss and propagation of regularity of transport equation with Osgood vector field. As an application, we obtain a quantitative stability estimate for 2D incompressible Euler equation with generalized Yudovich initial data.} 

\section{Introduction}
We consider the transport equation
\begin{equation} \label{transport}
\left \{
\begin{split}
&\partial_t \theta + u \cdot \nabla_x \theta = 0, \\
&\theta(0, \cdot) = \theta_0,
\end{split}
\right.
\end{equation}
where $u$ is a divergence-free field. It is well-known that if $u$ is has an Osgood modulus of continuity, that is, 
\begin{equation}
|u(x, t ) - u(t, y)| \le C \varphi(|x-y|),
\end{equation} 
for some increasing function $\varphi: (0, m) \rightarrow (0, \infty)$ satisfying
\begin{equation}
\begin{split}
&\lim_{r\rightarrow 0^+} \varphi(r) = 0, \\
&\int_0 ^m \frac{dr}{\varphi(r) } = \infty,
\end{split}
\end{equation}
then the associated ODE for the flow
\begin{equation} \label{ODE}
\left \{
\begin{split}
\frac{d}{dt} \phi(x,t) &= u(\phi(x,t), t), \\
\phi(x,0) &= x
\end{split} \right.
\end{equation}
has the unique solution (e.g. \cite{BCD2011}, \cite{MR1929104}) and there is unique weak integrable solution for \eqref{transport}, given by
\begin{equation}
\theta(x,t) = \theta_0 (\phi^{-1} (x,t) )
\end{equation}
(\cite{MR2439520}, \cite{MR4286465})
In that regard, in terms of existence and uniqueness of solution all Osgood modulus of continuities are equally good: they generate unique flows. \newline
What about the propagation of regularity? Suppose that $\theta_0 \in \dot{H}^\sigma (\mathbb{R}^d)$ for some $\sigma>0$. Is $\theta(t) \in \dot{H}^\sigma(\mathbb{R}^d)$? If $u$ is not Lipschitz, the answer is negative in general and there is loss of regularity. However, the degree of loss differs depending on the regularity of $u$. If $u$ is log-Lipschitz or better, we have losing estimates, that is, $\theta(t) \in \dot{H}^{\sigma(t) } (\mathbb{R}^d)$, for $\sigma(t) >0$ (e.g. \cite{BCD2011}). On the other hand, if $u \in \dot{W}^{1,p}$, $p<\infty$, one can immediately lose all Sobolev regularities: in \cite{ACM2018}, it has been shown that there is a divergence-free field $u \in \cap_{1 \le p < \infty} \dot{W}^{1,p}$ and $\theta_0 \in H^{\sigma} (\mathbb{R}^d)$ such that $\theta(t) \notin \dot{H}^{s} (\mathbb{R}^d)$ for every $s>0$. On the other hand, logarithm of a derivative can be preserved (e.g. \cite{MR2369485}, \cite{BN2018}, \cite{MS2022})

Osgood vector fields $u$ which are worse than log-Lipschitz lie between these two well known cases, and one may naturally ask the following questions. First, does Sobolev regularity propagate, even in a losing manner? Second, is there any kind of regularity, better than logarithm of a derivative, which is propagated? In this paper, we attempt to answer these questions.

First, we show that Sobolev regularity will not be propagated in general.
\begin{theorem}\label{main1}
Let $d \ge 2$ and $0<\sigma<\frac{d}{2}$. For each admissible growth function $\Theta$ (see definition \ref{admissible}), there exists a divergence-free vector field $u$ with a modulus of continuity $\varphi_\Theta$ (see definition \ref{assmod}) and $\theta_0 \in {H}^{\sigma} (\mathbb{R}^d)$ such that $\theta(t) \notin {H}^s (\mathbb{R}^d)$ for every $t>0, s>0$. 
\end{theorem}
As $\varphi_\Theta$ is an Osgood modulus of continuity (Lemma \ref{Osgood}), we see that all Sobolev regularity may loss immediately for Osgood velocity fields. Also, one may find $\theta_0 \in H^{\sigma} (\mathbb{R}^d)$ for any $\sigma > 0$: see Remark \ref{anysigma}.

Second, we show that certain regularities propagate, which depend on the modulus of continuity of $u$.  
\begin{theorem} \label{main2}
Suppose that $u$ is a divergence-free vector field with an Osgood modulus of continuity $\varphi$, and let $\mu_0$ be a modulus of continuity. 
\begin{enumerate}
\item Suppose that $\theta_0$ has modulus of continuity $\mu_0$, and $\theta(t)$ be the solution of \eqref{transport}. Then $\theta(t)$ has a modulus of continuity $\mu_0 \circ \mu_{[u]_\varphi, t}$ for $t \ge 0$ (see \eqref{mumod} for the definition of $\mu_{[u]_\varphi, t}$).
\item Let $T>0$ and suppose that $A(u) := \int_0 ^1 \frac{ \mu_0 \circ \mu_{[u]_{\varphi}, T } (r) }{r} dr < \infty$. Suppose that there exists an $g \in L^2 (\mathbb{R}^d) $ such that \eqref{initmod} holds. Then for $t \in [0, T]$,
\begin{equation}
\int_{B_1 (0) } \int_{\mathbb{R}^d} \frac{| \theta(x, t) - \theta(x+h, t)|^2 }{|h|^d  } \mu_0 \circ \mu_{[u]_\varphi, t } (|h| )^{-1} dx dh \le C \| g \|_{L^2 (\mathbb{R}^d) }^2  A(u).
\end{equation}
\end{enumerate}
\end{theorem}

As an application, we obtain a quantitative stability result for 2D Euler's equation in generalized Yudovich class with additional regularity for initial data, which improves that of \cite{KL2022}. Let $n \in \mathbb{Z}_{\ge1}$, and $\Theta_n: \mathbb{R}_{>0} \rightarrow \mathbb{R}_{>0}$ with $$\Theta_n(z) =  \log z \log_2 z \cdots \log_{n-1} z $$ for sufficiently large $z>1$. Let $\omega_{0,1}, \omega_{0,2} \in Y_{\Theta_n} (\mathbb{R}^2)$ with $u_{0,i} = \nabla^\perp \Delta^{-1} \omega_i \in L^2(\mathbb{R}^2)$ for $i=1,2$, where
\begin{equation}
Y_{\Theta_n} (\mathbb{R}^2) = \left \{ f \in \cap_{1 \le p < \infty} L^p(\mathbb{R}^2 ) | \| f \|_{L^p (\mathbb{R}^2) } \le \Theta_n(p) \right \}, \qquad \| f \|_{Y_{\Theta_n} } = \sup_p \frac{\|f\|_{L^p} }{\Theta_n(p) }.
\end{equation} 
Furthermore, suppose that $\omega_{0,i} \in H^s(\mathbb{R}^2)$ for some $0< s \le 1$, $i=1,2$. Let $\omega_i$, $i=1, 2$ be solutions of the initial-value problem
\begin{equation}
\left \{
\begin{split}
&\partial_t \omega_i + u_i \cdot \nabla_x \omega_i = 0, \\
&u_i = \nabla^\perp \Delta^{-1} \omega_i, \\ 
&\omega_i (0,x) = \omega_{0, i} (x).
\end{split} \right.
\end{equation}
\begin{theorem}\label{Euler}
If $\| u_{0,1} - u_{0,2} \|_{L^2} $ is sufficiently small, then there exist constants $C, C_1, C_2, M>0$, $\gamma < 0$ such that 
\begin{equation} \label{stabilityestimate}
\| \omega_1 (t) - \omega_2 (t) \|_{L^2} ^2 \le C \left ( \mu_{\omega, n, t}   \left ( \left (2 + \frac{ C \max_i \| \omega_{0,i} \|_{L^2}^2 }{  \mu_{\omega, n, t} ( \| u_{0,1} - u_{0,2} \|_{L^2} ^2 ) } \right )^\gamma \right ) \right )^s
\end{equation}
where $\mu_{\omega, n,t}$ is defined by
\begin{equation} \label{modcontivelvor}
\mu_{\omega,n, t} (r) := \frac{C_1}{e_{n-1} \left ( \left ( \log_{n-1} (C_2/r) \right )^{1/\exp(t M \max_i \| \omega_{0,i}\|_{Y_{\Theta_n} }  ) } \right ) } .
\end{equation}
\end{theorem}

\begin{remark}
Theorem \eqref{Euler} recovers the rate obtained in Proposition 18 of \cite{KL2022} for Yudovich case ($n=1$) and improves the rate for localized Yudovich case. In the case $n=1$, which is Yudovich case $\omega \in L^\infty$, $\mu_{\omega, n, t} ^i (r) = r^{C(t)}$, by Remark \ref{explicit}. Thus, $\| \omega_1 (t) - \omega_2 (t) \|_{L^2} ^2 \le C \|u_{0,1} - u_{0,2} \|_{L^2} ^{C(t) }$, which is the algebraic rate obtained in \cite{KL2022} (see also \cite{CDE2019}). On the other hand, when $n > 1$, we note that the result in \cite{KL2022} gives the rate of $1/\log(\mu_{\omega, n, t} ^2 (\| u_{0,1} - u_{0,2} \|_{L^2} ^2 ) )$. Since $\mu_{\omega, n, t} ^1$ vanishes faster than  logarithmic function by Lemma \ref{lognlipschitzreg}, the estimate \eqref{stabilityestimate} gives a smaller upper bound.
\end{remark}

\paragraph{Acknowledgement} The author greatly thanks Theo Drivas and In-Jee Jeong for stimulating discussion. Also, the author is deeply grateful to Gianluca Crippa and Anna Mazzucato for encouraging and helpful comments. 

\section{Immediate loss of Sobolev regularity}
\subsection{Growth functions and Osgood modulus of continuity}
\begin{definition} \label{admissible}
A function $\Theta: [1, \infty) \rightarrow (0, \infty)$ is called an admissible growth function if $\Theta$ is a differentiable function which is increasing with $\lim_{x\rightarrow \infty} \Theta(x) = \infty$, and there is a constant $M>1$, $C>0$ such that 
\begin{enumerate}
\item $\Theta(xy) \le \Theta(x) + \Theta(y) + C$ for all $x, y \ge M$, and
\item $x \Theta '(x) \le C \Theta(x)$ for all $x>M$.
\end{enumerate}
\end{definition}

\begin{remark}
The last two conditions are technical: they are satisfied for many slowly growing functions, and they are useful in obtaining certain bounds, but it would be interesting to obtain the same result with relaxed conditions. 
\end{remark}

An admissible growth function can grow slowly.
\begin{proposition} \label{logm}
$\Theta(x) = \log_m x, x > e_m (2) $, which is defined recursively by $\log_{m+1} x = \log_m (\log x)$, with $\log_1 x = \log x$, and $e_m = \log_m ^{-1}$ for $m \in \mathbb{N}$ is an admissible growth function ($\Theta(x)$ for $x \le e_m(2)$ can be chosen appropriately.)
\end{proposition}
\begin{proof}
It suffices to show the last two conditions. The last condition holds immediately: for $x > e_m(2)$, $\Theta'(x) = \prod_{j=1} ^{m-1} \frac{1}{\log_j (x) } \frac{1}{x}$, so $\Theta'(x) x \le C < C \Theta(x)$. For the condition $\Theta(xy) \le \Theta(x) + \Theta(y) + C$, we use induction on $m$: for $m=1$, $\Theta(x) = \log x$ and the condition obviously holds. Suppose that For $\log_m x$ the condition hold, with $M = M_m, C = C_m$. Then for $m+1$, for $x, y \ge e^{M_m} = M_{m+1}$, we have
\begin{equation}
\begin{split}
\Theta(xy) &= \log_{m} (\log x + \log y ) \le \log_m \big ( (M_{m}+2) \log \max (x,y) \big ) \\
&\le C_m +  \log_m (M_m+2) + \log_{m+1} \max (x,y)  \\
& \le (C_m +  \log_m (M_m+2)) + \log_{m+1} x + \log_{m+1} y = \Theta( x )+ \Theta( y) + C_{m+1}.
\end{split}
\end{equation}
\end{proof}

Also, an admissible growth function cannot grow too fast.
\begin{proposition}
There exists a constant $C'>0$ such that $\Theta(x) \le C'(\log x + 1)$ for all $x \ge 1$. 
\end{proposition}
\begin{proof}
This is immediate from the last property.
\end{proof}

As a consequence, $\Theta$ is associated with an Osgood modulus of continuity. We use the following modulus of continuity, following \cite{CS2021}.
\begin{lemma} \label{Osgood}
Let $\Theta$ be an admissible growth function. The function $\varphi_\Theta: [0, \infty) \rightarrow [0, \infty)$ defined by
\begin{equation} \label{Eulermodconti}
\varphi_\Theta (r) = \begin{cases} r \log (e/r) \Theta (\log (e/r) ),  0 < r < e^{-2}, \\ e^{-2}3 \Theta(3), r > e^{-2} \end{cases}
\end{equation}
is an Osgood modulus continuity.
\end{lemma}
\begin{proof}
\begin{equation}
\int_0 ^{e^{-2} } \frac{d r}{\varphi_\Theta (r) } = \int_3 ^\infty \frac{dp}{p \Theta(p) } \ge \int_3 ^\infty \frac{dp}{C' p (\log p + 1) } = \infty.
\end{equation}
\end{proof}

\begin{definition}\label{assmod}
For an admissible growth function $\Theta$, $\varphi_\Theta$ defined in lemma \ref{Osgood} is called the Osgood modulus of continuity associated with $\Theta$. 
\end{definition}
As the growth of $\Theta$ becomes slower, $\varphi_\Theta$ gets closer to log-Lipschitz. Since the class of admissible growth functions contains plenty of slowly growing functions (Proposition \ref{logm}), one can find $\varphi_\Theta$ which is close to log-Lipschitz. \newline
Next, we show that if $L^p$ norm of $\nabla u$ grows mildly, then $u$ is Osgood.
\begin{lemma}[\cite{CS2021}] \label{genYud}
Suppose that $\Theta(p)$ is an admissible growth function and there exists a constant $C>0$ such that $\| \nabla u \|_{L^p (\mathbb{R}^d)} \le C p \Theta(p)$ for all $p \ge 1$. Then $u$ has a modulus of continuity $\varphi_{\Theta}$.
\end{lemma}
\begin{proof}
Let $p > 2d$, and we apply Morrey's inequality:
\begin{equation}
|u(x) - u(y) | \le C_d \| \nabla u \|_{L^p (\mathbb{R}^d) } |x-y| ^{1 - \frac{d}{p} },
\end{equation}
where $C_d$ depends only on the dimension $d$. For $0 < |x-y| \le e^{-2}$, take $p = 1 - \log |x-y| \ge 3$. Then one immediately can see that
\begin{equation}
\begin{split}
|u(x) - u(y) | &\le C_d \log \left ( \frac{e}{|x-y| } \right ) \Theta \left ( \log \left ( \frac{e}{|x-y| } \right ) \right )|x-y|^{1 - \frac{d}{\log e/|x-y| } } \\
& \le C \log \left ( \frac{e}{|x-y| } \right ) \Theta \left ( \log \left ( \frac{e}{|x-y| } \right ) \right )|x-y| =C \varphi_{\Theta} (|x-y|),
\end{split}
\end{equation}
since $\lim_{r\rightarrow 0^+} r^{-d/\log(e/r) } = e^d$.
\end{proof}

\subsection{Alberti-Crippa-Mazzucato construction}
Our proof crucially relies on the construction made in \cite{ACM2018}, so we briefly summarize it.
\begin{proposition}[\cite{ACM2018_2}] Let $d\ge2$, $Q = \left [ -\frac{1}{2}, \frac{1}{2} \right ]^d$. There exist a divergence-free velocity field $v \in C^\infty (\mathbb{R}_{\ge 0} \times \mathbb{R}^d)$ and a corresponding non-trivial solution $\rho \in C^\infty(\mathbb{R}_{\ge 0} \times \mathbb{R}^d)$ of the continuity equation \eqref{transport} such that
\begin{enumerate}
\item for every $t \ge 0$, $v(t, \cdot)$ and $\rho(t, \cdot)$ are bounded and supported in $Q$,
\item there is a constant $B>0$, independent of $p$ such that $ |\nabla v (t, \cdot) \|_{L^p (\mathbb{R}^d)} \le B$ for all $1 \le p \le \infty$, and
\item for every $0 < s <2$ there exist constants $C_s, c>0$ such that for every $t \ge 0$, $\|\rho(t, \cdot ) \|_{\dot{H}^s (\mathbb{R}^d) } \ge C_s e^{sct}$.
\end{enumerate}
\end{proposition}

Next, we consider a sequence of cubes $Q_n$ in $\mathbb{R}^n$ of side-length $\lambda_n \searrow 0$ and center $q_n \in \mathbb{R}^d$ which are disjoint to each other, contained in a compact set, and converging to the origin. Next, we set a sequence $\tau_n \searrow 0$ and in each cube $Q_n$ define
\begin{equation}
u_n (t,x) = \frac{\lambda_n}{\tau_n} v \left (\frac{t}{\tau_n}, \frac{x - q_n}{\lambda_n} \right ), \theta_n (t,x) = \rho\left (\frac{t}{\tau_n}, \frac{x-q_n}{\lambda_n} \right ).
\end{equation} 
We note the following: for detailed computations one may refer to \cite{ACM2018}.
\begin{enumerate}
\item $\theta_n$ solves $\partial_t \theta_n + u_n \cdot \nabla_x \theta_n = 0.$
\item The distance between support of $\theta_n$ and $Q_n^c$ is at least $\lambda_n$,
\item $\| u_n \|_{\dot{W}^{1,p} (\mathbb{R}^d) } \le \frac{\lambda_n^{\frac{d}{p } }}{\tau_n}$, $\|\theta_n (0) \|_{\dot{H}^\sigma(\mathbb{R}^d) } \le \lambda_n^{\frac{d}{2} - \sigma}$, and $\| \theta_n (t) \|^2 _{\dot{H}^s (\mathbb{R}^d) } \ge   \lambda_n^{d-2s} \left (C_s^2 \exp \left (\frac{2 s c t }{\tau_n } \right ) - \frac{C}{s} \right )$ for some constant $C$, for every $0<s<1$.
\end{enumerate}

The key lemma is the following:
\begin{lemma}\label{key}
There exists $\lambda_n \searrow 0$, $\tau_n \searrow 0$ such that
\begin{enumerate}
\item $\sum_n \lambda_n < \infty$, and thus $Q_n$ can be chosen to be disjoint, contained in a compact set and converging to the origin, 
\item there exists a constant $C>0$ such that  $\sum_n \frac{\lambda_n^{\frac{d}{p} } }{\tau_n } \le C p \Theta(p)$ for all $1 \le p <\infty$, so that $u = \sum_{n} u_n$ is in $\dot{W}^{1,p}$ with $\| u \|_{\dot{W}^{1,p} } \le C p \Theta(p)$,
\item there exists $C>0$ such that $\frac{\lambda_n}{\tau_n} \le C$ for every $n$, so that $u(t, \cdot) \in L^\infty(\mathbb{R}^d)$ uniformly in $t$,
\item $\sum_n \lambda_n ^{\frac{d}{2} - \sigma} < \infty$, so that $\theta_0 = \sum_n \theta_n(0, \cdot) \in \dot{H}^\sigma(\mathbb{R}^d)$,
\item $\sum_n \lambda^{d-2s} \exp \left (\frac{2 s c t }{\tau_n } \right ) = \infty$ for every $0<s<1$ and $t>0$, so that $\theta(t, \cdot ) = \sum_{n} \theta_n(t, \cdot) \notin \dot{H}^{s}(\mathbb{R}^d)$.
\end{enumerate}
\end{lemma}
Note that $\theta(t, \cdot ) \in L^\infty (\mathbb{R}^d)$ uniformly in $t$. Also, since each $Q_n$ are disjoint to each other, and $\cup_n Q_n$ is contained in a compact set $K$, $\sum_{k=1} ^n u_k \rightarrow u, \sum_{k=1} ^n \theta_k \rightarrow \theta$ strongly in $L^2 ([0, T] \times K)$ for every $T>0$. Thus, for any $\psi \in C^\infty_0 ((0,\infty) \times \mathbb{R}^d)$, $0 = \int \partial_t \psi \sum_{k=1} ^n \theta_k +  \nabla_x \psi \cdot (\sum_{k=1} ^n u_k ) \sum_{k=1} ^n \theta_k \rightarrow \int \partial_t \psi \theta + \nabla_x \psi \cdot u \theta$, that is, $\theta$ is a weak solution of \eqref{transport} with $u$. Moreover, as $u$ is Osgood by Lemma \ref{Osgood}, $\theta$ is the unique integrable weak solution. Since $\theta(t,\cdot) \notin \dot{H}^s (\mathbb{R}^d)$ for any $t >0$ and any $0<s<1$, while $\theta(t, \cdot) \in L^2 (\mathbb{R}^2)$ uniformly in time, this proves the Theorem \ref{main1}.
\begin{remark} \label{anysigma}
One can extend Theorem \ref{main1} to $\sigma \ge \frac{d}{2}$ immediately: one may introduce another factor $\gamma_n \searrow 0$ and set $\theta = \sum_n \gamma_n \theta_n$. In fact, along with $\lambda_n, \tau_n$, setting $\gamma_n = \lambda_n ^\sigma$ works.
\end{remark}

\subsection{Proof of Lemma \ref{key}}
Let $\Theta$ be an admissible growth function. We set
\begin{equation}
\lambda_n = e^{-e^n}, \tau_n^{-1}  = \log (1/\lambda_n) \Theta(\log(1/\lambda_n ) ) = e^n \Theta(e^n).
\end{equation}
Obviously, condition 1, 3, and 4 of Lemma \ref{key} are satisfied. For condition 5, we see that
\begin{equation}
\begin{split}
\sum_n \lambda^{d-2s} \exp \left ( \frac{2sct} {\tau_n} \right ) &= \sum_n \exp \left ( (2s-d) \log (1/\lambda_n) + \log (1/\lambda_n) \Theta (\log (1/\lambda_n ) ) \right ) \\
&\ge C + \sum_{n \ge N} \exp ( C \log(1/\lambda_n) \Theta(\log(1/\lambda_n ) ) )  = \infty.
\end{split}
\end{equation}
It remains to verify condition 2. Since
\begin{equation}
\sum_n \frac{\lambda_n ^{\frac{d}{p} } }{\tau_n} = \sum_{n=1 } ^\infty e^{-\frac{d}{p} e^n } e^n \Theta(e^n) = \sum_{n=1} ^\infty F(e^n),
\end{equation}
where $F(x) = x \Theta(x) e^{-\frac{d}{p} x}$. Note that $\lim_{x\rightarrow \infty} F(e^x) = 0$, and $F$ is bounded; moreover, either there is $\bar{x}_p \ge 1$ such that $F(e^x)$ attains maximum at $\bar{x}_p$, $F(e^x)$ is increasing on $[1, \bar{x}_p]$, and decreasing on $(\bar{x}_p, \infty)$ or $F$ is decreasing on $\mathbb{R}_{\ge 1}$. Therefore, comparing the series with integrals, we have
\begin{equation} \label{integraltest}
\sum_{n=1} ^\infty F(e^n) \le \int_{1} ^\infty F(e^x) dx + 2F(e^{\bar{x}_p} ),
\end{equation}
where the last term of \eqref{integraltest} (if exists) comes from that for $n = 1, \cdots, [\bar{x}_p]-1$, one has $F(e^n) \le \int_{n} ^{n+1} F(e^x) dx$, and for $n = [\bar{x}_p]+2, \cdots$ one has $F(e^n) \le \int_{n-1} ^n F(e^x) dx$, while at $n = [\bar{x}_p], [\bar{x}_p]+1$, one has $F(e^n) \le F(e^{\bar{x}_p})$. We estimate each term of \eqref{integraltest}.
\paragraph{Estimate of $F(e^{\bar{x}_p} )$.} Note that $F(e^{\bar{x}_p} ) = \| F \|_{L^\infty}$, and it is attained either at the unique point $y$ such that $F'= 0$ or $y=1$. Since $F'(y) = ( \Theta(y) + y \Theta'(y) - \frac{d}{p} y \Theta(y) ) e^{-\frac{d}{p}y}$, we have
\begin{equation}
\Theta(y) + y \Theta'(y) = \frac{d}{p} y \Theta(y).
\end{equation} 
Since $\Theta$ is admissible, we have $(C+1) \Theta(y) \ge \frac{d}{p} y \Theta(y)$, that is, $y \le \frac{(C+1) }{d} p.$ Therefore, 
\begin{equation}
\begin{split}
\| F \|_{L^\infty} &= F(y) \le \frac{C+1}{d} p \Theta \left  ( \frac{C+1}{d} p  \right) \le C' p \Theta \left ( \left (M+\frac{C+1}{d} \right ) p \right) \\
& \le C' p \left ( \Theta (p) + \Theta  \left ( \left (M+\frac{C+1}{d} \right )  \right) + C \right ) \le C'' p \Theta(p)
\end{split}
\end{equation}
for large enough $p$.
\paragraph{Estimate of $\int_1 ^\infty F(e^x) dx$} We note that, by change of variable $z = e^x$,
\begin{equation}
\int_1 ^\infty F(e^x) dx = \int_1 ^\infty \Theta(e^x) e^{-\frac{d}{p} e^x } e^x dx = \int_e ^\infty\Theta(z) e^{-\frac{d}{p} z } dz.
\end{equation}
By putting $y = \frac{d}{p}z$, we have
\begin{equation}
\int_e ^\infty\Theta(z) e^{-\frac{d}{p} z } dz = \frac{p}{d} \int_{\frac{de}{p} } ^\infty \Theta \left (\frac{p}{d} y \right ) e^{-y} dy.
\end{equation}
Note that for $p > M$, $\Theta \left (\frac{p}{d} y \right ) \le \Theta(pM/d) \le \Theta (pM) \le \Theta(p) +\Theta(M) + C$ if $y \le M$, and if $y > M$ $\Theta\left (\frac{p}{d} y \right) \le \Theta(py) \le \Theta(p) + \Theta(y) + C$.  In any case, we have a constant $C'>0$ satisfying
\begin{equation}
\Theta \left (\frac{p}{d} y \right ) \le \Theta(p) + \Theta(y) + C', y \ge 1.
\end{equation}
Also, if $\frac{de}{p} \le y \le 1$, we have $\Theta \left (\frac{p}{d} y \right ) \le \Theta (p) $. Thus, we have
\begin{equation}
\begin{split}
 \frac{p}{d} \int_{\frac{de}{p} } ^\infty \Theta \left (\frac{p}{d} y \right ) e^{-y} dy &\le \frac{p}{d} \int_{\frac{de}{p}} ^1 \Theta(p) e^{-y} dy + \frac{p}{d} \int_1 ^\infty (\Theta(p) + \Theta(y) + C' ) e^{-y} dy \\
 & \le C'' p \Theta(p) 
\end{split}
\end{equation}
for some constant $C''>0$, due to the integrability of $\Theta(y) e^{-y}$.

\section{Propagation of regularity}
Still stronger regularity is propagated than $\nabla u \in L^p$ case: in the latter case, only logarithm of derivatives are propagated, provided that the initial data of \eqref{transport} $\theta_0$ has one derivative (see, for example, \cite{BN2018}.) 
\subsection{Propagation of modulus for $\theta$}
In this subsection, we establish propagation of modulus of continuity result, based on observations made in \cite{DEL2022}. We start with a result proved in \cite{DEL2022}.
\begin{lemma}[\cite{DEL2022}]
Suppose that $\varphi$ is an Osgood modulus of continuity for $u$, and define
\begin{equation}
\mathcal{M}(z) := \int_z ^m \frac{dr}{\varphi(r) },\qquad  R(z) := \exp ( - \mathcal{M}(z) ),
\end{equation}
and
\begin{equation}
[u(t) ]_{\varphi} := \sup_{x\ne y} \frac{|u(x,t) - u(y,t) | }{\varphi(|x-y| )}.
\end{equation}
Then $R$ is monotonely increasing, $\lim_{z\rightarrow 0^+} R(z) = 0$, and 
\begin{equation} \label{mumod}
|\phi^{-1} (x,t) - \phi^{-1} (y,t) | \le R^{-1} \left ( e^{\int_0 ^t [u(s) ]_{\varphi} ds } R(|x-y| ) \right ) =: \mu_{[u]_\varphi, t} (|x-y| ),
\end{equation}
where $\phi(x,t)$ is the unique flow generated by $u$: i.e. the unique solution of \eqref{ODE} and $\phi^{-1} (x,t) $ is its back-to-label map.
\end{lemma}
\begin{proof}
This is an immediate consequence of Osgood's lemma and time-reversal.
\end{proof}
An immediate consequence is the following.
\begin{lemma}
Let $\mu_0$ be a modulus of continuity, and let $g$ be a measurable function satisfying
\begin{equation} \label{initmod}
|\theta_0 (x) - \theta_0 (y) | \le (g(x) + g(y)) \mu_0 (|x-y|).
\end{equation}
Then
\begin{equation} \label{tmod}
|\theta(x,t) - \theta(y,t) |  \le (g_t (x) + g_t (y) ) \mu_0 ( \mu_{[u]_\varphi, t} (|x-y| ) )
\end{equation}
for $g_t = g \circ \phi_t ^{-1} $.
\end{lemma}
\begin{proof}
We may simply recall that $\phi^{-1}$ is volume preserving and
\begin{equation}
\begin{split}
|\theta(x,t) - \theta(y,t) | &= | \theta_0 (\phi^{-1} (x,t) ) - \theta_0 (\phi^{-1} (y,t) ) | \\
& \le (g(\phi^{-1}(x,t) ) + g(\phi^{-1} (y,t) ) ) \mu_0 (| \phi^{-1} (x,t) - \phi^{-1} (y,t) | ) \\
& \le (g(\phi^{-1}(x,t) ) + g(\phi^{-1} (y,t) ) ) \mu_0 (\mu_{[u]_\varphi, t} (|x-y| ) ).
\end{split}
\end{equation}
\end{proof}
This immediately gives the first part of Theorem \ref{main2}.
\begin{remark} \label{explicit}
It is worth noting that the modulus of continuity $\mu_{0, \varphi} = R$ plays a special role: if $\theta_0$ has a modulus of continuity $R$, then $\theta(t)$ still has the same modulus of continuity $R$ since $\mu_{0, \varphi} \circ \mu_{[u]_\varphi, t} (r) = e^{\int_0 ^t [u(s) ]_{\varphi} ds} R(r) = e^{\int_0 ^t [u(s) ]_{\varphi} ds} \mu_{0, \varphi} (r).$ 
\end{remark}
\begin{remark} The following are explicit calculation of $\mu_{[u]_\varphi, t} (r)$ for some $\varphi$.
\begin{itemize}
\item If $u$ is Lipschitz, that is, $\varphi(z) = z$, then $\mathcal{M}(z) = \log (1/z)$ for sufficiently small $z$, $R(z) = z = R^{-1} (z)$, and thus $\mu_{[u]_\varphi, t} (r) = e^{\int_0 ^t [u(s) ]_{\varphi} ds } r.$ Thus, if $\theta_0$ is Lipschitz or Holder, then $\theta(t)$ retains the same modulus of continuity.
\item If $u$ is log-Lipschitz, that is, $\varphi(z) = z \log (1/z)$, then $\mathcal{M}(z) = \log \log(1/z)$ for sufficiently small $z$, $R(z) = \frac{1}{\log (1/z) }$, $R^{-1} (z) = e^{-1/z}$, and thus $\mu_{[u]_\varphi, t} (r) = r ^{\frac{1}{\exp(\int_0 ^t [u(s) ]_\varphi ds ) } }.$ Thus, if $\theta_0$ is Lipschitz or Holder, then $\theta(t)$ remains Holder, with its Holder exponent decreasing over time.
\item If $\varphi (z) = z \log (1/z) \log_2 (1/z) \cdots \log_n (1/z)$ for sufficiently small $z$, then $\mathcal{M}(z) = \log_{n+1} (1/z)$, $R(z) = \frac{1}{\log_n (1/z)}$, $R^{-1}(z) = \frac{1}{e_n(1/z)}$ where $e_n$ is the inverse of $\log_n$. Therefore, $\mu_{[u]_\varphi, t} (r) = \frac{1}{e_{n-1} \left ( \left ( \log_{n-1} (1/r) \right )^{1/\exp(\int_0^t [u(s) ]_\varphi ds ) } \right ) }$. 
\item If $\varphi(z) = (1-\alpha) z^\alpha$, $\alpha \in (0,1)$, then $\mathcal{M}(z) = 1-z^{1-\alpha}$, $R(z) = e^{-1} e^{z^{1-\alpha} }$, $R^{-1}(z) = \frac{1}{1-\alpha} \log (ez)$, and $\mu_{[u]_{\varphi, t} } (r) = \frac{1}{1-\alpha} \left (\int_0 ^t [u]_{\varphi, t} (s) ds + r^{1-\alpha } \right )$. In particular, $\mu_{[u]_{\varphi, t} }$  does not vanish as $r\rightarrow 0$, and therefore $\mu_0 \circ \mu_{[u]_{\varphi, t} }$ is not a modulus of continuity in general. 
\end{itemize}
Regarding the vanishing rate of $\mu_{[u]_{\varphi, t} } (r) $ for $\varphi (z) = z \log (1/z) \log_2 (1/z) \cdots \log_n (1/z)$, we have the following result.
\end{remark}
\begin{lemma} \label{lognlipschitzreg}
Let $\varphi_n (z) = z \log (1/z) \log_2 (1/z) \cdots \log_n (1/z)$ and define 
\begin{equation} \label{modun}
\mu_{n, t} (r) := \frac{1}{e_{n-1} \left ( \left ( \log_{n-1} (1/r) \right )^{1/\exp(\int_0^t [u(s) ]_\varphi ds ) } \right ) }.
\end{equation}
\begin{itemize}
\item $\mu_{n, t} $ modulus of continuity is worse than any Holder modulus of continuity: $\lim_{r\rightarrow 0^+} \frac{\mu_{n, t} (r) }{r^\alpha} = \infty$ for any $\alpha \in (0, 1]$ for $n \ge 2$.
\item $\mu_{n, t} $ modulus of continuity is better than any logarithmic modulus: $\lim_{r\rightarrow 0^+} \frac{\mu_{n, t } (r) }{ 1/(\log (1/r) )^a } = 0$ for any $a \ge 1$.
\end{itemize}
\end{lemma}
\begin{proof}
For the first claim, it suffices to show that $\lim_{r\rightarrow 0^+} \frac{ \log_{n-1} 1/r^\alpha} {(\log_{n-1} 1/r)^\beta} = \infty$ for $\alpha >0$ sufficiently small and $\beta \in (0, 1)$. For $n= 2$, it obviously holds. For $n\ge 3$, as $\log_{n-2} x$ is an admissible growth function, $\log_{n-1} (1/r) \le \log_{n-1} (1/r^\alpha) + \log_{n-2} 1/\alpha + C$ for sufficiently small $\alpha>0$ and a constant $C>0$. Thus, $\frac{ \log_{n-1} 1/r^\alpha} {(\log_{n-1} 1/r)^\beta} \ge \frac{\log_{n-1} 1/r^\alpha }{(\log_{n-1} (1/r^\alpha) + C' )^\beta} \rightarrow \infty$ as $r\rightarrow 0^+$. 

For the second claim, it suffices to show that $\lim_{r\rightarrow 0^+} \frac{ \log_{n-1} (\log(1/r) )^a }{(\log_{n-1} (1/r) )^\beta } = 0$ for any $\beta \in (0,1)$ and $a\ge 1$ sufficiently large, and again it is obvious for $n=2$. For $n \ge 3$, since $\log_{n-2}$ is an admissible growth function, we have $\log_{n-1} (\log(1/r) )^a \le \log_{n-2} a + \log_n (1/r) + C$ for sufficiently large $a \ge1$ and a constant $C>0$, and the desired estimate follows.
\end{proof}
\subsection{Propagation of Sobolev-type functionals}
Finally, we prove the second part of Theorem \ref{main2}. Our argument is inspired by that of \cite{BN2018}. We use the layer-cake representation, \eqref{tmod}, and monotonicity of $\mu_{[u]_\varphi, t}$ on $t$, and volume-preserving property of the flow $\phi_t^{-1}$ to obtain
\begin{equation}
\begin{split}
& \int_{\mathbb{R}^d} {| \theta(x, t) - \theta(x+h, t)|^2 }  dx = 2 \int_0 ^\infty r \left | \left \{ x : | \theta(x, t) - \theta(x+h, t) | > r \right \} \right | dr \\
&\le 2 \int_0 ^\infty r \left | \left \{ x : g_t (x) + g_t (x+h) > r / \mu_0 \circ \mu_{[u]_\varphi, T} (|h| ) \right \} \right | dr \\
&= 2 (\mu_0 \circ \mu_{[u]_\varphi , T} (|h| ))^2 \int_0^\infty r | \{ x: g_t (x) + g_t (x+h) > r \} | dr \\
& \le 2 (\mu_0 \circ \mu_{[u]_\varphi , T} (|h| ))^2 \int_0 ^\infty r | \{x:  2 g_t (x) > r \} | dr \le C (\mu_0 \circ \mu_{[u]_\varphi , T} (|h| ))^2 \| g \|_{L^2(\mathbb{R}^d)} ^2.
\end{split}
\end{equation}
Therefore, we have
\begin{equation}
\begin{split}
&\int_{B_1 (0) } \int_{\mathbb{R}^d} \frac{| \theta(x, t) - \theta(x+h, t)|^2 }{|h|^d  } \mu_0 \circ \mu_{[u]_\varphi, t } (|h| )^{-1} dx dh \\
&\le  C\| g \|_{L^2(\mathbb{R}^d) } ^2 \int_{B_1 (0) } \frac{\mu_0 \circ \mu_{[u]_\varphi , T} (|h| )}{|h|^d} dh \le C \| g \|_{L^2 (\mathbb{R}^d) } ^2 \int_0 ^1 \frac{\mu_0 \circ \mu_{[u]_\varphi , T} (r ) }{r} dr = C A(u) \| g \|_{L^2(\mathbb{R}^d) } ^2.
\end{split}
\end{equation}
\begin{remark}
In the case of $\varphi(z) = z \log (1/z) \log_2 (1/z) \cdots \log_n (1/z)$, by Lemma \ref{lognlipschitzreg}, $\mu_{[u]_\varphi, t} = o ( 1/\log(1/r) ^a ) $ for any $a \ge 1$: therefore we have
\begin{equation}
\int_{B_1 (0)} \int_{\mathbb{R}^d} \frac{|\theta(x,t) - \theta(x+h, t) |^2}{|h|^d} \mu_{[u]_\varphi, t} (|h|)^{-1} dx dh \ge C \int_{B_1 (0) }\int_{\mathbb{R}^d} \frac{|\theta(x,t) - \theta(x+h, t) |^2}{|h|^d (\log(1/|h|))^{1-p} } dx dh 
\end{equation}
for some $C>0$ and any $p >1$. In this sense, transport equation with Osgood velocity field $u$ with modulus of continuity $\varphi$ propagates more regularity than logarithm of a derivative.
\end{remark}

\section{Stability result for 2D Euler equation}
\subsection{An interpolation lemma}
We start with an interpolation lemma, which is analogous to Proposition 3.5 of \cite{BN2018}.
\begin{lemma}\label{interpolation}
Let $\varepsilon \in (0,1)$, and let $\mu:(0, 1) \rightarrow \mathbb{R}_{>0}$ be a continuous increasing function. Then there exists a constant $C>0$ such that the following holds.
\begin{equation} \label{interpolation}
\begin{split}
C^{-1} \| f \|_{L^2 (\mathbb{R}^2 )}^2 &\le  \mu (\varepsilon ) \int_{B_1 (0) } \int_{\mathbb{R}^d} \frac{|f(x+h) - f(x) |^2}{|h|^d } \mu (|h| ) ^{-1} dx dh \\
& + |\log \varepsilon | \frac{\|f\|_{L^2 (\mathbb{R}^2 ) } ^2 }{\log \left (  2 + \frac{\|f \|_{L^2(\mathbb{R}^d ) } ^2  }{\| f \|_{\dot{H}^{-1} (\mathbb{R}^d )  } ^2  } \right ) } .
\end{split}
\end{equation}
\end{lemma}
\begin{proof}
We closely follow the proof of Proposition 3.5 of \cite{BN2018}. We first fix $\delta \in (0,1)$, $\psi \in C^\infty_c (\mathbb{R}^d)$, a radial function, with $\psi \equiv 1$ in $B_{2/3} \setminus B_{1/2}$, $\psi \equiv 0$ in $(B_{5/6} \setminus B_{1/3} )^c$, and $\int \psi_{\mathbb{R}^d} = 1$, and write $\psi_\delta (x) =\delta^{-d} \psi(x/\delta)$. We decompose $f$ by
\begin{equation}
f = f*\psi_\delta + (f - f*\psi_\delta)
\end{equation}
and treat each term separately. We have
\begin{equation}
\begin{split}
\| f - f*\psi_\delta \|_{L^2}^2 &= \int_{\mathbb{R}^d}  \left | \int_{\mathbb{R}^d}  (f(x+h) - f(x) ) \psi_\delta (|h|) dh \right |^2 dx \\
&\le \int_{\mathbb{R}^d} \int_{\delta/3 \le |h| \le \delta}  \frac{|f(x+h)-f(x)|^2}{\delta^d} dh \| \psi\|_{L^2} ^2 dx \\
&\le C \int_{\delta/3 \le |h| \le \delta} \int_{\mathbb{R}^d} \frac{|f(x+h) - f(x) |^2}{|h|^d } dx dh, \\
\|f * \psi_\delta \|_{L^2} ^2 &= \|  \hat{\psi_\delta} \hat{f} \|_{L^2} ^2 \le \sup_{\xi} | \log (2+|\xi|) | |\hat{\psi} (\delta \xi ) |^2 \int_{\mathbb{R}^d} \frac{|\hat{f}(\xi)|^2}{\log (2+|\xi | ) } d\xi \\
&\le \max(\sup_{|\xi|\le 2} 2 \log 2 |\hat{\psi} (\delta \xi)|^2 , \sup_{|\xi|>2} 2 \log |\xi| |\hat{\psi} (\delta \psi )|^2 )\int_{\mathbb{R}^d} \frac{|\hat{f}(\xi)|^2}{\log (2+|\xi | ) } d\xi \\
&\le C \log (\max (1/2,1/\delta ) ) \int_{\mathbb{R}^d} \frac{|\hat{f}(\xi)|^2}{\log (2+|\xi | ) } d\xi.
\end{split}
\end{equation}
For the treatment of $ \int_{\mathbb{R}^d} \frac{|\hat{f}(\xi)|^2}{\log (2+|\xi | ) } d\xi$, we split the integral by low and high frequencies and optimize:
\begin{equation}
\begin{split}
 \int_{\mathbb{R}^d} \frac{|\hat{f}(\xi)|^2}{\log (2+|\xi | ) } d\xi &=  \int_{|\xi| \le \nu} \frac{|\xi|^2}{\log (2+|\xi | ) }|\xi|^{-2} |\hat{f}(\xi)|^2 d\xi +  \int_{|\xi| > \nu} \frac{|\hat{f}(\xi)|^2}{\log (2+|\xi | ) } d\xi \\
 & \le  \frac{\nu^2}{\log (2+\nu ) } \int_{|\xi| \le  \nu} |\xi|^{-2}|\hat{f}(\xi)|^2  d\xi + \frac{1}{\log(2+\nu)} \int_{|\xi| \ge \nu} |\hat{f}(\xi)|^2 d \xi \\
 &\le C \frac{\nu^2}{\log (2+\nu^2 ) } \| f \|_{\dot{H}^{-1} }^2 +\frac{C}{\log(2+\nu^2) } \| f \|_{L^2} ^2,
\end{split}
\end{equation}
and taking $\nu = \| f\|_{L^2} / \| f \|_{\dot{H}^{-1} }$ gives 
\begin{equation}
 \int_{\mathbb{R}^d} \frac{|\hat{f}(\xi)|^2}{\log (2+|\xi | ) } d\xi \le \frac { C \| f \|_{L^2} ^2 }{\log \left (  2 + \frac{\|f \|_{L^2(\mathbb{R}^d ) }^2  }{\| f \|_{\dot{H}^{-1} (\mathbb{R}^d )  }^2  } \right )}.
\end{equation} 
Thus, we have
\begin{equation} \label{twoparts}
\begin{split}
\| f \|_{L^2} ^2 &\le  C \int_{\delta/3 \le |h| \le \delta} \int_{\mathbb{R}^d} \frac{|f(x+h) - f(x) |^2}{|h|^d } dx dh \\
& + C\log (\max (1/2,1/\delta ) )\frac { \| f \|_{L^2} ^2 }{\log \left (  2 + \frac{\|f \|_{L^2(\mathbb{R}^d ) }^2  }{\| f \|_{\dot{H}^{-1} (\mathbb{R}^d )  }^2  } \right )}.
\end{split}
\end{equation}
Integrating $\delta$ over $(\varepsilon, 1)$ with weight $\frac{\mu(\delta)^{-1}}{\delta}$ gives us the following:
\begin{equation}
\begin{split}
\| f \|_{L^2}^2 & \le C \frac{1}{\int_\varepsilon ^1 \frac{\mu(\delta)^{-1} }{\delta } d \delta } \int_{\varepsilon} ^1 \int_{\delta/3 \le |h| \le \delta} \int_{\mathbb{R}^d} \frac{|f(x+h) - f(x) |^2}{|h|^d } dx dh \frac{\mu(\delta)^{-1}}{\delta} d \delta \\
&+ C \frac{1}{\int_\varepsilon ^1 \frac{\mu(\delta)^{-1} }{\delta } d \delta } \int_\varepsilon ^1 \log (\max (1/2,1/\delta ) )\frac { \| f \|_{L^2} ^2 }{\log \left (  2 + \frac{\|f \|_{L^2(\mathbb{R}^d ) }^2  }{\| f \|_{\dot{H}^{-1} (\mathbb{R}^d )  } ^2 } \right )} \frac{\mu(\delta)^{-1}}{\delta} d \delta \\
& \le C \frac{\int_\varepsilon ^1 \frac{d\delta}{\delta } } {\int_\varepsilon ^1 \frac{\mu(\delta)^{-1} }{\delta } d \delta } \int_{B_1 (0) } \int_{\mathbb{R}^d} \frac{|f(x+h) - f(x) |^2}{|h|^d } dx dh  + C | \log \varepsilon |  \frac { \| f \|_{L^2} ^2 }{\log \left (  2 + \frac{\|f \|_{L^2(\mathbb{R}^d ) } ^2 }{\| f \|_{\dot{H}^{-1} (\mathbb{R}^d )  }^2  } \right )} .
\end{split}
\end{equation}
Since $\mu$ is increasing, we have $\frac{\int_\varepsilon ^1 \frac{d\delta}{\delta  } } {\int_\varepsilon ^1 \frac{\mu(\delta)^{-1} }{\delta } d \delta } \le \mu(\varepsilon)$, and \eqref{interpolation} follows.
\end{proof}
\subsection{Proof of Theorem \ref{Euler}}
From Lemma \ref{Osgood} and Lemma \ref{genYud}, we see that $u_1, u_2$ have modulus of continuity $\varphi_{\Theta_n}$ as in \eqref{Eulermodconti}, with 
\begin{equation}
[u_i ]_{\varphi_{\Theta_n} } \le M \| \omega_{i,0} \|_{Y_{\Theta_n} }, i=1,2.
\end{equation}
Therefore, there is $C_1, C_2>0$ such that
\begin{equation} \label{mu1}
\mu_{[u_i]_{\varphi_{\Theta_n } , t} } (r) \le \mu_{\omega, n, t}  (r), i=1,2.
\end{equation}
Next, we have the following stability estimate of velocity for Euler equation:
\begin{proposition}[\cite{KL2022}]
For some constant $C_1, C_2>0$, we have
\begin{equation} \label{mu2}
\| u_1 (t) - u_2 (t) \|_{L^2} ^2 \le \mu_{\omega, n, t}  ( \| u_{0,1} - u_{0,2} \|_{L^2} ^2 ) ,
\end{equation}
\end{proposition}
One can adjust $C_1, C_2>0$ appropriately (if necessary) so that \eqref{mu1} and \eqref{mu2} are simultaneously satisfied. Next, by the standard Lusin estimate, we have
\begin{equation}
\| \omega_{0,i} (x) - \omega_{0,i} (y) | \le C |x-y|^s (D_s \omega_{0,i}(x) + D_s \omega_{0,i} (y) ), i=1,2,
\end{equation}
where
\begin{equation}
D_s f(x) := \left (\int_{\mathbb{R}^d}\frac{|f(x+h) - f(x) |^2}{|h|^{d+2s} } dh \right)^{\frac{1}{2} } \in L^2.
\end{equation}
Therefore, by the same argument as in Theorem \ref{main2}, we have
\begin{equation}
\int_{B_1 (0) } \int_{\mathbb{R}^d} \frac{|\omega_i (x,t) - \omega_i (x+h, t) |^2}{|h|^d} (\mu_{\omega, n, t} (|h|))^{-s} dxdh \le C \| D_s \omega_i \|_{L^2} ^2 \int_0 ^1 \frac{(\mu_{\omega, n, t} (r))^{s}}{r} dr <\infty,
\end{equation}
since $\mu_{\omega, n, t}$ vanishes faster than any logarithm by Lemma \ref{lognlipschitzreg}. Therefore, applying  \eqref{interpolation} with $\mu(r)  = (\mu_{\omega, n, t} (r))^s$ and $f = \omega_1 - \omega_2$ gives that for sufficiently small $\varepsilon>0$, we have
\begin{equation}
\begin{split}
C^{-1} \| \omega_1 (t) - \omega_2 (t) \|_{L^2}^2 &\le  (\mu_{\omega, n, t} (\varepsilon))^s + |\log \varepsilon | \frac{\| \omega_1 (t) - \omega_2 (t) \|_{L^2} ^2 }{\log \left ( 2 + \| \omega_1 (t) - \omega_2 (t) \|_{L^2} ^2 / \| u_1 (t) - u_2 (t) \|_{L^2} ^2 \right ) } \\
&\le (\mu_{\omega, n, t} (\varepsilon))^s + |\log \varepsilon| \frac{\| \omega_1 (t) - \omega_2 (t) \|_{L^2} ^2 }{\log \left ( 2 +\frac{ \| \omega_1 (t) - \omega_2 (t) \|_{L^2} ^2}{\mu_{\omega, n, t} (\| u_{0,1} - u_{0,2} \|_{L^2} ^2 ) } \right ) } .
\end{split}
\end{equation}
Now take 
\begin{equation}
\varepsilon = \left ( 2 +\frac{ \| \omega_1 (t) - \omega_2 (t) \|_{L^2} ^2}{\mu_{\omega, n, t} (\| u_{0,1} - u_{0,2} \|_{L^2} ^2 ) } \right )^\gamma
\end{equation}
for some $\gamma<0$ so that $$ \frac{|\log \varepsilon | }{\log \left ( 2 +\frac{ \| \omega_1 (t) - \omega_2 (t) \|_{L^2} ^2}{\mu_{\omega, n, t} (\| u_{0,1} - u_{0,2} \|_{L^2} ^2 ) } \right ) }   < C^{-1}/2,$$
and then we have
\begin{equation}
\| \omega_1 (t) - \omega_2 (t) \|_{L^2}^2 \le C \mu_{\omega, n, t}  \left ( \left ( 2 +\frac{ \| \omega_1 (t) - \omega_2 (t) \|_{L^2} ^2}{\mu_{\omega, n, t} (\| u_{0,1} - u_{0,2} \|_{L^2} ^2 ) } \right )^\gamma \right ).
\end{equation}
Replacing $ \| \omega_1 (t) - \omega_2 (t) \|_{L^2} ^2$ on the right-hand side by $2 \max_i \| \omega_{0,i} \|_{L^2}$ gives \eqref{stabilityestimate}.

\bibliographystyle{abbrv}
\bibliography{Bibliography} 

\begin{thebibliography}{10}

\bibitem{ACM2018_2}
G.~Alberti, G.~Crippa, and A.~L. Mazzucato.
\newblock Exponential self-similar mixing by incompressible flows.
\newblock {\em J. Amer. Math. Soc.}, 32(2):445--490, 2019.

\bibitem{ACM2018}
G.~Alberti, G.~Crippa, and A.~L. Mazzucato.
\newblock Loss of regularity for the continuity equation with non-{L}ipschitz
  velocity field.
\newblock {\em Ann. PDE}, 5(1):Paper No. 9, 19, 2019.

\bibitem{MR2439520}
L.~Ambrosio and P.~Bernard.
\newblock Uniqueness of signed measures solving the continuity equation for
  {O}sgood vector fields.
\newblock {\em Atti Accad. Naz. Lincei Rend. Lincei Mat. Appl.},
  19(3):237--245, 2008.

\bibitem{BCD2011}
H.~Bahouri, J.-Y. Chemin, and R.~Danchin.
\newblock {\em Fourier analysis and nonlinear partial differential equations},
  volume 343 of {\em Grundlehren der mathematischen Wissenschaften [Fundamental
  Principles of Mathematical Sciences]}.
\newblock Springer, Heidelberg, 2011.

\bibitem{BN2018}
E.~Bru\'{e} and Q.-H. Nguyen.
\newblock Sharp regularity estimates for solutions of the continuity equation
  drifted by {S}obolev vector fields.
\newblock {\em Anal. PDE}, 14(8):2539--2559, 2021.

\bibitem{MR4286465}
L.~Caravenna and G.~Crippa.
\newblock A directional {L}ipschitz extension lemma, with applications to
  uniqueness and {L}agrangianity for the continuity equation.
\newblock {\em Comm. Partial Differential Equations}, 46(8):1488--1520, 2021.

\bibitem{CDE2019}
P.~Constantin, T.~D. Drivas, and T.~M. Elgindi.
\newblock Inviscid limit of vorticity distributions in the {Y}udovich class.
\newblock {\em Comm. Pure Appl. Math.}, 75(1):60--82, 2022.

\bibitem{MR2369485}
G.~Crippa and C.~De~Lellis.
\newblock Estimates and regularity results for the {D}i{P}erna-{L}ions flow.
\newblock {\em J. Reine Angew. Math.}, 616:15--46, 2008.

\bibitem{CS2021}
G.~Crippa and G.~Stefani.
\newblock An elementary proof of existence and uniqueness for the euler flow in
  localized yudovich space.
\newblock {\em arxiv preprint.}

\bibitem{DEL2022}
T.~D. Drivas, T.~Elgindi, and J.~La.
\newblock Propagation of singularities by osgood vector fields and for 2d
  inviscid incompressible fluids.
\newblock {\em arxiv preprint.}

\bibitem{MR1929104}
P.~Hartman.
\newblock {\em Ordinary differential equations}, volume~38 of {\em Classics in
  Applied Mathematics}.
\newblock Society for Industrial and Applied Mathematics (SIAM), Philadelphia,
  PA, 2002.
\newblock Corrected reprint of the second (1982) edition [Birkh\"{a}user,
  Boston, MA; MR0658490 (83e:34002)], With a foreword by Peter Bates.

\bibitem{KL2022}
C.~Kim and J.~M. La.
\newblock Vorticity convergence from boltzmann to 2d incompressible euler
  equations below yudovich class.
\newblock {\em arxiv preprint.}

\bibitem{MS2022}
D.~Meyer and C.~Seis.
\newblock Propagation of regularity for transport equations. a littlewood-paley
  approach.
\newblock {\em arxiv preprint.}

\end{thebibliography}
\end{document}